\documentclass[11pt]{amsart}
\usepackage{amscd,amsfonts,amsmath,amssymb,amstext,amsthm}
\usepackage[all]{xy}

\newcommand{\RR}{\mathbb{R}} 
\newcommand{\CC}{\mathbb{C}} 

\theoremstyle{plain}
\newtheorem{theorem} {Theorem} [section]
\newtheorem{lemma} [theorem]{Lemma}
\newtheorem{proposition}[theorem]{Proposition}
\newtheorem{corollary} [theorem]{Corollary}
\theoremstyle{definition}

\newtheorem{remark} [theorem]  {Remark}

\numberwithin{equation}{section}

\title{Biholomorphisms between Hartogs domains over homogeneous Siegel domains}
\date\today

\author{Aeryeong Seo}
\address{
School of Mathematics, Korea Institute for Advanced Study (KIAS),
 85 Hoegiro (Cheongnyangni-dong 207-43), Dongdaemun-gu,
 Seoul 130-722, Republic of Korea
phone number: 82-2-958-2636
}
\email{aeryeongseo@kias.re.kr}


\begin{document}

\maketitle

\begin{abstract}
In this paper, we characterize the Hartogs domains over homogeneous Siegel domains
of type II and explicitly describe their automorphism groups.
Moreover we prove that any proper holomorphic map between Hartogs domains over
homogeneous Siegel domains over type II is a biholomorphism.
\end{abstract}

\section{Introduction}
Let $D$ be a domain in the complex Euclidean space $\mathbb C^r$.
Let $K_D\colon D\times \overline D\rightarrow \mathbb C$ be the Bergman kernel of $D$, i.e.,
$K_D$ is the reproducing kernel of the Hilbert space of
holomorphic $L^2$ functions on $D$.
Suppose that $K_D(z,z)$ is positive for any $z\in D$.
For a positive real number $s\geq 1$ and a positive integer $N$, 
let $D_{N,s}$ be the Hartogs domain over $D$ defined by
\begin{equation}\label{Hartogs domain}
D_{N,s} = \left\{(z,\zeta)\in D\times \mathbb C^N : ||\zeta||^2 < K_D(z,z)^{-s}\right\},
\end{equation}
where $||\cdot||$ denotes the standard Hermitian norm on $\mathbb C^N$.

The purpose of this paper is to prove the following two theorems:
\begin{theorem}\label{s=s'}
Let $D$ and $D'$ be homogeneous Siegel domains of type II .
Suppose that there exists a biholomorphism $f\colon D_{N,s}\rightarrow D'_{N',s'}$.
Then $D$ and $D'$ are biholomorphic, $N=N'$, $s=s'$
and $f$ is of the form $(f_1, f_2)$ with
\begin{equation}\label{induced auto}
f_1(z,\zeta) = \phi(z), \quad f_2(z,\zeta) = U(\zeta)\,\left( J\phi(z)\right)^s
\end{equation}
for some biholomorphism $\phi$ from $D$ to $D'$ 
and some unitary transformation $U$.
\end{theorem}
\begin{corollary}\label{automorphism group}
Let $D$ be a homogeneous Siegel domain of type II.
Then the set of all biholomorphic self-maps 
$\text{Aut}(D_{N,s})$ consists of the maps of the form \eqref{induced auto}
with any $\phi\in \text{Aut}(D)$ and any unitary transformation $U$ in $\mathbb C^n$.
\end{corollary}
\begin{theorem}\label{proper map}
Let $f\colon D_{N,s} \rightarrow D'_{N',s'}$ be a proper holomorphic map between
equidimensional Hartogs domains over homogeneous Siegel domains of type II
which are not biholomorphic to the unit ball.
Then $f$ is a biholomorphism.
\end{theorem}

The motivation of this paper is to generalize the work of Ahn-Byun-Park in \cite{Ahn_Byun_Park}
and Tu-Wong in \cite{Tu_Wang}.
\begin{theorem}[Ahn, Byun and Park \cite{Ahn_Byun_Park}]
Let $D$ be a bounded symmetric domain of classical type which is not biholomorphic to the unit ball.
Then the set of all biholomorphic self-maps 
$\text{Aut}(D_{N,s})$ consists of holomorphic automorphisms of the form 
$\Phi=(\Phi_1, \Phi_2) \colon D_{N,s}\rightarrow D_{N,s}$ with
\begin{equation}
\Phi_1(z,\zeta) = \phi(z), \quad \Phi_2(z,\zeta) = U(\zeta)\,\left( J\phi(z)\right)^s
\end{equation}
for any $\phi\in \text{Aut}(D)$ and unitary transformation $U$.
\end{theorem}
In their original work, Tu and Wang studied Hua domains over bounded symmetric 
domains of classical type and hence their result is more general. 
However for simplicity of exposition, we only state their theorems in case of Hartogs domains.
\begin{theorem}[Tu and Wang \cite{Tu_Wang}]
\begin{enumerate}
\item
Let $D$ and $D'$ be irreducible bounded symmetric domains of classical type.
Let $f\colon D_{N,s}\rightarrow D'_{N',s'}$ be a biholomorphism.
Then $N=N'$, $s=s'$ and there exist an automorphism $\Phi$ of $D'_{N',s'}$, 
a complex linear isomorphism $A$ of $\mathbb C^{\dim_{\mathbb C} D}$ 
and a unitary transformation $U$ of $\mathbb C^N$ such that 
\begin{equation}\label{form}
\Phi\circ f(z,\zeta) = (Az, U\zeta).
\end{equation}
In particular any automorphism of $D_{N,s}$ is of the form \eqref{induced auto}.
\item
Let $f\colon D_{N,s} \rightarrow D'_{N',s'}$ be a proper holomorphic map between
equidimensional Hartogs domains over bounded symmetric domains of classical type.
Then $f$ is a biholomorphism.
\end{enumerate}

\end{theorem}

When $D$ is the unit ball, $D_{N, s}$ is biholomorphic to a generalized ellipsoid.
In this case  Naruki \cite{Naruki} showed that every biholomorphism has the form \eqref{form}
and Kodama \cite{Kodama} proved that every automorphism has the form \eqref{induced auto}.

The automorphism group of Hartogs domains over minimal homogeneous domain
in $\mathbb C^n$ with center $0$ was described in \cite{Zhao_Wang_Hao}.
Note that every bounded homogeneous domain can be 
realized as a homogeneous Siegel domain of type II.
One of the difficulties in proving Corollary \ref{automorphism group}
for general bounded homogeneous domains is linked with fact that  
$D_{N,s}$ might not be circular if $D$ is not a bounded symmetric domain.
Here one says that a domain is circular if
$(z_1, z_2,\ldots, z_n)\in D$ implies that also 
$(e^{i\theta}z_1, e^{i\theta}z_2, \ldots, e^{i\theta}z_n)$
belongs to the domain for every $\theta \in \mathbb R$.
By a work of Bell \cite{Bell}, if the bounded domain is circular,
every proper holomorphic map between equidimensional domains can be extended 
holomorphically over the boundary. 
Hence we can use the boundary structure of the domains.
Moreover, if $D$ is a circular domain containing $0$, 
then the automorphism of $D$ fixing $0$ is a linear map
(cf. \cite{Greene_Kim_Krantz} Corollary 1.3.2.).

Recently, in \cite{Ishi_Park_Yamamori}, the Bergman kernel of the domain 
\eqref{Hartogs domain} over the bounded homogeneous domain was explicitly obtained
by Ishi-Park-Yamamori.
With this explicitly expressed Bergman kernel, we can exploit the Tumanov's method in \cite{Tumanov}. In this paper he proved the following:
\begin{theorem}[Tumanov \cite{Tumanov}]
Let $D$ and $D'$ be equidimensional Siegel domains of type II.
Then any proper holomorphic map $f\colon D\rightarrow D'$ is biholomorphic and rational.
\end{theorem}
For the proof of it, 
he presented that the proper holomorphic maps between equidimensional domains 
can be locally extended over the boundary 
under some nice conditions on the Bergman kernel of the source and target domains.
For more detail see \cite{Tumanov} or Lemma \ref{Tumanov} below.

The internal organization of the article is as follows.
In Section \ref{preliminaries}, we recall the homogeneous Siegel domains of type II
and their Bergman kernels. The explicit expression of the Bergman kernel of $D_{N,s}$ 
over the homogeneous Siegel domain of type II will be given.
In Section \ref{automorphism group for general domains} we describe the group 
of automorphisms of the form \eqref{induced auto} 
and show that the boundary points that belong to 
$\{ (z,\zeta)\in D\times \mathbb C^N : ||\zeta||^2 = K_D(z,z)^{-s}\}$ 
are strongly pseudoconvex for general domains 
when the Bergman metric is positive definite.
In Section \ref{The proof} we provide the proof of Theorem \ref{s=s'},
Corollary \ref{automorphism group} and Theorem \ref{proper map}.

\bigskip
{\bf Acknowledgement}
The author would like to thank Ngaiming Mok 
for the discussions on this subject and useful comments.

\section{preliminaries}\label{preliminaries}
\subsection{Homogeneous Siegel domains of type II and their properties}
In this section, we recall the definition and properties of homogeneous Siegel domains of type II
that we need later. For more detail, see \cite{Kaneyuki,Pya}.

A domain $V$ in $\mathbb R^n$ is called a {\it convex cone} if the following three conditions are satisfied:
\begin{enumerate}
\item For any $x\in V$ and for any $\lambda >0$, $\lambda x \in V$,
\item If $x,\,y\in V$, then $x+y \in V$,
\item $V$ contains no entire straight lines.
\end{enumerate}
A convex cone $V$ is called {\it homogeneous} if the linear automorphism group, $G(V) = \{ A\in GL(\RR^n) : AV = V\}$, acts transitively on $V$. 
A map $F \colon \mathbb C^m \times \mathbb C^m \rightarrow \mathbb R^n$ is called a {\it $V$-Hermitian form} if the following conditions are satisfied:
\begin{enumerate}
\item $F(u,v)$ is $\mathbb C$-linear in $u$,
\item $\overline{F(u,v)} = F(v,u)$,
\item $F(u,u) \in \overline{V}$ where $\overline{V}$ is the closure of $V$ in $\RR^n$,
\item $F(u,u)=0$ implies $u=0$.
\end{enumerate}
The {\it Siegel domain of type II} is defined by
\begin{equation} \label{Siegel domain}
D(V,F) = \{(w,u) \in \CC^n \times \CC^m : \text{Im}{w} - F(u,u) \in V\}.
\end{equation}
When $\text{Aut}(D(V,F))$ acts transitively on $D(V,F)$, it is called {\it homogeneous}
and when the affine linear group of $\mathbb C^n\times \mathbb C^m$ which 
preserves $D(V,F)$ acts transitively on $D(V,F)$, it is called {\it affine homogeneous}.

\begin{theorem}[Pyateskii-Shapiro \cite{Pya}]
\begin{enumerate}
\item
Every homogeneous bounded domain in the complex Euclidean space 
is holomorphically equivalent to a homogeneous Siegel domain of type II.

\item The followings are equivalent: $(i)$ $D(V,F)$ is homogeneous, 
$(ii)$ $D(V,F)$ is affine homogeneous, $(iii)$ $V$ is homogeneous. 
\item
Let $G_a$ be the affine automorphism groups of a homogeneous Siegel domain of type II $D(V,F)$.
Let $H$ be the isotropy subgroup of $G_a$ at $(0,0)\in \mathbb C^n\times \mathbb C^m$
and $K_a$ the isotropy subgroup of $G_a$ 
for some fixed $(ir,0)\in \mathbb C^n\times \mathbb C^m$ with $r\in \mathbb R^n$.
Let $\frak g_a$, $\frak k_a$ be the Lie algebras of $G_a$, $K_a$ respectively.
Then one has $G_a = H\cdot \mathbb R^n \mathbb C^m$, $K_a\subset H$ and 
\begin{equation}
\frak g_a = \frak k_a + j\mathbb R^n +\mathbb R^n +\mathbb C^n
\end{equation}
where $j$ is a linear endomorphism of $\frak g_a$ induced from the complex structure of $D(V,F)$.

\item  
Define $\frak g = j\mathbb R^n +\mathbb R^n +\mathbb C^n$ which is a normal $j$-algebra
corresponding to $D(V,F)$ (for more detail, see \cite{Pya, Kaneyuki}).
Then there exist an $1$-dimensional ideal $\frak r$, a subalgebra $\frak Z$
and a normal $j$-subalgebra $\frak g_1\subset \frak g$ 
such that $\frak g= \frak r + j\frak r +\frak Z + \frak g_1$ 
satisfying 
\begin{equation}
[\frak r +j\frak r, \frak g_1]=0, \quad [\frak Z, \frak g_1]\subset \frak Z.
\end{equation}
Remark that the normal $j$-subalgebra $\frak r + j\frak r + \frak Z$
 corresponds to the Siegel domain of type II which is biholomorphic to the unit ball.
Besides $\frak r + j\frak r$ corresponds to the upper half plane in $\mathbb C$.
\end{enumerate}
\end{theorem}

\begin{lemma}\label{product}
Let $D$ be a homogeneous Siegel domain of type II which is not biholomorphic to
the unit ball.
Then there is a complex manifold $M$ such that $\Delta\times M$ is holomorphically 
and properly embedded in $D$.
\end{lemma}
\begin{proof}
Let $\frak g$ be a normal $j$-algebra corresponding to $D$ and $G$ its Lie group.
Then there exist a $j$-subalgebra $\frak r + j\frak r +\frak Z$ and 
a normal $j$-subalgebra $\frak g_1$ 
such that $\frak g= \frak r + j\frak r +\frak Z + \frak g_1$ and $[\frak r +j\frak r, \frak g_1]=0$.
Let $R$ and $G_1$ be corresponding connected Lie subgroups 
of $\frak r$ and $\frak g_1$ respectively in $G$.
Then the $R$-orbit of the base point of $D$ is biholomorphic to the unit disc and 
$G_1$-orbit is a complex submanifold $M$ such that the product of the $R$-orbit and $M$ can be 
properly and holomorphically embedded in $D$.
\end{proof}

For a Siegel domain of type II,
$D = \{ (w,u)\in \mathbb C^n\times \mathbb C^m : \text{Im}(w) - F(u,u)\in V\}$, 
the Shilov boundary of $D$ is given by 
\begin{equation}
S = \{ (w,u)\in \mathbb C^n\times \mathbb C^m : \text{Im}(w) = F(u,u)\}.
\end{equation}
Let $D'=\{((w,u)\in \mathbb C^n\times \mathbb C^m : \text{Im}(w) - F'(u,u)\in V'\}$ be 
another Siegel domains of type II and $S'$ its Shilov boundary.
\begin{lemma}\label{Shilov boundary}
Let $f\colon D\rightarrow D'$ be a biholomorphism. Then $f$ maps the Shilov boundary of $D$
into that of $D'$. 
\end{lemma}
\begin{proof}
Suppose that $f(S)\not\subset S'$.
Then there is $p\in S$ and $q\in \partial D'\setminus S'$ such that $f(p) = q$.
Let $\mu$ be a holomorphic function on some open neighborhood of $\overline D$ 
which attains $\sup_{\overline D} |\mu|$ only at $p$.
Then $G\circ f^{-1}$ attains the supremum at $q$. 
This contradicts to the definition of the Shilov boundary.
\end{proof}

For a given homogeneous Siegel domain $D(V,F)$ of type II,
the Bergman kernel has the form
\begin{equation}\label{Bergman kernel of Siegel}
\begin{aligned}
K_{D(V,F)}((w,u), (w', u'))& = c_{V,F}
\left( \frac{u - \overline{u'}}{2\sqrt{-1}}-F(w, w')\right)^{2d-q}\\
&=b\left(\frac{u - \overline{u'}}{2\sqrt{-1}}-F(w, w')\right)
\end{aligned}
\end{equation}
where $c_{V,F}$ is a constant independent of $(z,\zeta)$ and $(z',\zeta')$
and multi-indices $d=(d_1,\ldots,d_R)$, $q=(q_1,\ldots, q_R)$ 
are given by the Lie algebra structure of affine automorphism group of $D(V,F)$.
In the second line, $b$ is a holomorphic function on the tube domain
$\{u \in \mathbb C^m : \text{Im} u\in V\}$ which is a product of rational functions.
Note that every $d_j$ is semi-negative integer 
and every $q_k$ is semi-positive integer. 
Moreover $2(d_1+\cdots + d_R) - (q_1+\cdots + q_R)<-2$.
For the detail, see \cite{Gindikin}.

\subsection{The Bergman kernel of $D_{N,s}$}
For a homogeneous Siegel domain $D$,
the Bergman kernel $K_{N,s}$ of $D_{N,s}$ is given by the following \cite[Theorem 4.3]{Ishi_Park_Yamamori}:
\begin{equation}\label{Bergman kernel of Hartogs}
K_{N,s}\left( (z,\zeta), (z', \zeta')\right) 
= \frac{N!}{\pi^N}K_{D(V,F)}(z,z')^{N+1}b\left(t \frac{d}{dt}+N\right) \Big[ (1-t)^{-m-1}\Big]_{t=K_{D(V,F)}(z,z')^s\langle\zeta, \zeta'\rangle}
\end{equation}
where $b$ is a polynomial. For more detail, see \cite[Section 4]{Ishi_Park_Yamamori}.

For a bounded domain $D$ in the Euclidean space, we say that the Bergman kernel $K_D$ of $D$
{\it extends} over the boundary if for every compact subset $E$ of $D$, 
there is an open neighborhood $U$ of $D$ such that $\overline E\subset U$ and 
$K(\cdot, p)$ is holomorphic on $U$ for every $p\in E$.
By the equations \eqref{Bergman kernel of Siegel} and \eqref{Bergman kernel of Hartogs},
we can obtain the following:
\begin{lemma}
Let $D$ be a Siegel domain of type II. Then 
the Bergman kernels of $D$ and $D_{N,s}$ extend over the boundary.
\end{lemma}

\section{Hartogs domains over general domains}
\subsection{Automorphisms of $D_{N,s}$}\label{automorphism group for general domains}
Let $D$ be a domain in $\mathbb C^r$ such that $K_D(z,z)\neq 0$ for any $z\in D$.
For an automorphism $\phi\in \text{Aut}(D)$, let $\Phi$ be a holomorphic mapping defined by
$\Phi=(\Phi_1, \Phi_2) \colon D_{N,s}\rightarrow D_{N,s}$ with
\begin{equation}
\Phi_1(z,\zeta) = \phi(z), \quad \Phi_2(z,\zeta) = U(\zeta)\,\left( J\phi(z)\right)^s.
\end{equation}
Here $J\phi$ denotes the Jacobian determinant of $\phi$ and $U$ is a unitary transformation.
Since \begin{equation} ||\Phi_2(z,\zeta)||^2 = ||\zeta||^2 |J\phi(z)|^{2s}
\end{equation} and 
\begin{equation}
K_D(\Phi_1(z,\zeta), \Phi_1(z,\zeta))^{-s} = \left( K_D(z,z) |J\phi(z)|^{-2}\right)^{-s},
\end{equation}
it follows that $\Phi$ is well-defined and the holomorphic map $\Phi$ is an automorphism of $D_{N,s}$.

\begin{remark}
Let $\phi \colon D\rightarrow D'$ be a biholomorphism between domains $D$ and $D'$.
Then $D_{N,s}$ and $D'_{N,s}$ are biholomorphic by the map
\begin{equation}\label{biholo map}
(z,\zeta)\mapsto \left( \phi(z), J\phi(z)^s \zeta\right)
\end{equation}
where $J\phi$ denotes the Jacobian determinant of $f$.
By the transformation formula of the Bergman kernel by biholomorphisms, it is easy
to show that \eqref{biholo map} is a biholomorphism.
\end{remark}

Let $G$ be the set of automorphisms induced by all $\phi \in \text{Aut}(D)$.

\begin{lemma}
$G$ is a subgroup of $\text{Aut}(D_{N,s})$.
\end{lemma}
\begin{proof}
Let $\phi,\, \psi\in \text{Aut}(D)$ and $\Phi,\, \Psi\in \text{Aut}(D_{N,s})$ induced by $\phi,\, \psi$ respectively.
Then 
\begin{equation}
\begin{aligned}
\Psi\circ \Phi (z,\zeta) &= \Psi\left( \Phi_1(z,\zeta), \Phi_2(z,\zeta)\right)
=\Psi\left( \phi(z), U(\zeta) (J\phi(z))^s\right)\\
&=\left( \psi\circ \phi(z), V( U(\zeta) (J\phi(z))^s) (J\psi)^s\circ(\phi(z))\right)\\
&= \left( \psi\circ \phi(z), VU(\zeta) (J(\psi\circ\phi)(z))^2\right)
\end{aligned}
\end{equation}
and hence the lemma is proved.
\end{proof}

\subsection{Strongly pseudoconvex boundary points}
Decompose the boundary of $D_{N,s}$ by 
\begin{equation}
\partial D_{N,s} = \partial_0 D_{N,s} \bigcup \left( \partial D\times \{0\}\right)
\end{equation} 
with \begin{equation}
\partial_0 D_{N,s} := \{ (z,\zeta)\in D\times \mathbb C^N : ||\zeta||^2 = K_D(z,z)^{-s}\}.
\end{equation}

\begin{lemma}\label{strongly pseudoconvex}
Suppose that the Bergman metric of $D$ is well-defined. Then 
any point in $\partial_0D_{N,s}$ is strongly pseudoconvex.
\end{lemma}
\begin{proof}
For the notational convention, we will write $K_D(z,z)$ by $K$.
Denote $\rho(z) =  ||\zeta||^2 - K_D(z,z)^{-s}$.
Note that 
\begin{equation}
\begin{aligned}
\partial\overline\partial \rho(z) &= 
\left(
\begin{array}{cc}
-\partial\overline\partial K^{-s} & 0 \\
0 & I_{N\times N}
\end{array}
\right),\\
 -\partial\overline\partial K^{-s} &= sK^{-s-1} \partial\overline\partial K
-s(s+1) K^{-s-2} \partial K\wedge \overline\partial K,  \\
\partial\overline\partial \log K &= \frac{\partial \overline\partial K}{K} - \frac{\partial K\wedge \overline\partial K}{K^2}.
\end{aligned}
\end{equation}
This implies that 
\begin{equation}
 -\partial\overline\partial K^{-s} = sK^{-s} \partial\overline\partial \log K
-s^2 K^{-s-2} \partial K\wedge \overline\partial K,  
\end{equation}
Let $V= X+Y$ with $X=\sum_{j=1}^r X_j\frac{\partial}{\partial z_j}$ and 
$Y=\sum^N_{j=1} Y_j\frac{\partial}{\partial \zeta_j}$ be a complex tangent vector
on the boundary, i.e., $\partial \rho (V)=\sum_{j=1}^N \overline\zeta_j Y_j + sK^{-s-1}\partial K(X)=0$.
Hence we obtain
\begin{equation}
|\partial K(X)|^2 = s^{-2}K^{2s+2}\Big|\sum^N_{j=1}\overline\zeta_j Y_j\Big|^2.
\end{equation}
Since $\partial \overline\partial \rho(V,\overline V) = ||Y||^2  - \partial\overline\partial K^{-s}(X,\overline X)$,
\begin{equation}
\begin{aligned}
\partial \overline\partial \rho(V, \overline V) 
&= ||Y||^2 + sK^{-s} \partial\overline\partial \log K(X, \overline X)
-s^2 K^{-s-2} \left( s^{-2}K^{2s+2}\Big|\sum^N_{j=1}\overline\zeta_j Y_j\Big|^2\right)\\
&=  sK^{-s} \partial\overline\partial \log K(X, \overline X) +||Y||^2 
- K^s\Big|\sum^N_{j=1}\overline\zeta_j Y_j\Big|^2\\
&=sK^{-s} \partial\overline\partial \log K(X, \overline X) + ||\zeta||^{-1}\left( ||Y||^2||\zeta||^2 
- \Big|\sum^N_{j=1}\overline\zeta_j Y_j\Big|^2\right) >0
\end{aligned}
\end{equation}
where the third equality is due to $||\zeta||^2 = K^{-s}$.
\end{proof}
\begin{remark} 
Lemma \ref{strongly pseudoconvex} holds Hartogs domain over a domain $D$ 
of the form \eqref{Hartogs domain} with $s>0$.
\end{remark}
\section{Proof of theorems}\label{The proof}
The first step that we need to consider is that the maps in the theorems send 
$D\times \{0\}$ to $D'\times \{0\}$.
When the map is an automorphism of the bounded symmetric domains of classical type,
it is prove in \cite{Ahn_Byun_Park}
and when the map is a proper holomorphic map between bounded symmetric domains of 
classical type, it is proved in \cite{Tu_Wang}.
In our case, we use the following lemma:

\begin{lemma}\cite[Proposition 2.2.]{Mok_Tsai}\label{radial extension}
Suppose that $f$ is a bounded holomorphic function defined on $\Delta\times W$ 
where $W$ is a bounded domain in $\mathbb C^k$.
Then there exists a holomorphic function $f_\theta$ on $W$ satisfying 
$\lim_{r\rightarrow 1}f(re^{i\theta}, w) = f_\theta(w)$ for almost every $\theta\in \partial \Delta$. 
Moreover, if $f_\theta$ is constant for a measurable subset of $\partial \Delta$,
then $f$ is independent of variables of $W$.
\end{lemma}
\begin{lemma}\label{center to center}
Let $D$ and $D'$ be homogeneous Siegel domains of type II and 
$f\colon D_{N,s}\rightarrow D'_{N', s'}$ a proper holomorphic map.
Suppose that $D$ and $D'$ are not biholomorphic to the unit ball.
Then $f$ maps $D\times \{0\}$ to $D'\times\{0\}$.
\end{lemma}

\begin{proof}
Suppose that $f(D\times \{0\})\not\subset D' \times \{0\}$.
Choose $z_0\in D$ such that $f(z_0,0)\in D'_{N',s'}\setminus D'\times \{0\}$.
Let $\imath \colon \Delta \times M \hookrightarrow D$ be a proper holomorphic embedding 
such that $\imath(0,y_0)=z_0$ for some $y_0\in M$, i.e.,
\begin{equation}\label{eq1}
f\circ\imath(0,y_0) \in D'_{N',s'}\setminus D'\times \{0\}.
\end{equation}
Choose $\theta \in [0,2\pi)$ such that 
$f\circ \imath(e^{i\theta}, \cdot) \colon M\rightarrow \mathbb C^{N'+r'}$ 
is a well- defined holomorphic map and 
$f\circ \imath(e^{i\theta}, M)\subset \partial D'_{N',s'}$.
Because of \eqref{eq1}, there exists $y \in M$ such that 
$f\circ \imath (e^{i\theta}, y)\in \partial_0 D'_{N',s'}$.
Since $\partial_0 D'_{N',s'}$ is strongly pseudoconvex, 
$f\circ\imath(e^{i\nu}, \cdot)$ is a constant map for each $\nu$ in some neighborhood
of $\theta$. 
Hence $f\circ \imath$ is independent of $M$ and it is a contradiction
to the fact that $f$ is proper.
Therefore $f(D\times\{0\})\subset D'\times \{0\}$.
\end{proof}


\begin{proposition}\label{rigidity of holo}
Let $U$ be an open neighborhood of $0\in \mathbb C^{r+N}$ and
$f\colon U\rightarrow \mathbb C^{r+N}$ a holomorphic map 
such that  
\begin{enumerate}
\item
$f( D_{N,s}\cap U)\subset D_{N,s'}$ and 
$f(\partial D_{N,s}\cap U)\subset \partial D_{N,s'}$ with $s\leq s'$, \label{condition 1}
\item
$f(U\cap (D\times \{0\}))\subset D\times \{0\}$ and 
$f|_{U\cap (D\times \{0\})}\equiv id$, 
\item $d_\zeta f|_{(0,0)}\neq 0$.\label{condition 3}
\end{enumerate}
Then $s=s'$ and the Jacobian matrix of $f$ at $0$ is $\left( \begin{array}{cc}
Id_r& b\\
0& c\\
\end{array}
\right)
$ for some unitary $N$ by $N$ matrix $c$ and $r$ by $N$ matrix $b$.

\end{proposition}
\begin{proof}
Denote $f = (f_1, f_2)$ where $f_1\colon D_{N,s}\rightarrow D$,
$f_2\colon D_{N,s}\rightarrow \mathbb C^{N}$ are projections
and $\rho_1(z,\zeta)= |\zeta|^2 - K_D(z,z)^{-s}$.
Fix a point $(z_0,\zeta_0)\in \partial D_{N,s}\cap U$ such that $\zeta_0\neq 0$. 
Note that $d\rho_1|_{(z_0,\zeta_0)}\neq 0$.
Since
$|f_2(z,\zeta)|^2 - K_D(f_1(z,\zeta), f_1(z,\zeta))^{-s'}=0$ whenever 
$z\in U$ and $|\zeta|^2 - K_D(z,z)^{-s}=0$, 
there exists a real analytic function $\mu(z,\zeta)$ on a sufficiently small
neighborhood of $(z_0,\zeta_0)$ such that 
\begin{equation}\label{equation}
|f_2(z,\zeta)|^2 - K_D(f_1(z,\zeta), f_1(z,\zeta))^{-s'}
 = \mu((z,\zeta),(\overline z,\overline \zeta))\left(|\zeta|^2 - K_D(z,z)^{-s}\right).
\end{equation}
By polarization, we obtain the holomorphic equation, 
\begin{equation}\label{polarization}
f_2(z,\zeta)\cdot\overline f_2(w,\xi) - K_D(f_1(z,\zeta), f_1(\overline w,\overline \xi))^{-s'}
 = \mu((z,\zeta),(w,\xi))\left(\zeta\cdot \xi - K_D(z,\overline w)^{-s}\right).
\end{equation}
By the identity property of holomorphic functions, it follows that 
\begin{equation}
f_2(z,\zeta)\cdot\overline f_2(w,\xi) - 
K_D(f_1(z,\zeta), f_1(\overline w,\overline \xi))^{-s'}=0
\end{equation}
whenever
$\zeta\cdot \xi - K_D(z,\overline w)^{-s}=0$ and $ ((z,\zeta),(w,\xi))\in U\times U$.
This implies that \eqref{equation} holds for all $(z,\zeta)\in U$ by substituting 
$(w,\xi)=(\overline z,\overline \zeta)$

Substitute $\zeta=0$ on \eqref{equation}. Then $$\mu(z,0) = K(z,z)^{s-s'}.$$
If we expand $\mu$ in terms of $\zeta$ variable, we can express $\mu$ as
\begin{equation}
\mu(z,\zeta) = K(z,z)^{s-s'} + \sum_{j=1}^N \zeta_j g_j(z,\overline z) 
+ \sum^N_{j=1} \overline \zeta_j \overline{g_j(z,\overline z)} + O(\zeta^2)
\end{equation}
with some functions $g_j$ where $O(\zeta^2)$ denotes the higher order terms
with respect to $\zeta$.
Note that $K(0,0)^{s-s'}=0$ if $s\neq s'$.
Let
\begin{equation}
\begin{aligned}
f_1(z,\zeta) &= z + b\zeta + O^2\\
f_2(z,\zeta) &= az + c\zeta + O^2
\end{aligned}
\end{equation}
with linear transformations $a, b, c$ 
where $O^2$ denotes the higher order terms in $\zeta$ and $z$.
Then the equation \eqref{equation} can be expressed by
\begin{equation}\label{equation2}
\begin{aligned}
&| az + c\zeta + O^2|^2 - c\left( \text{Im}(\pi_1(z + b\zeta + O^2))- F(\pi_2(z + b\zeta + O^2), \pi_2(z + b\zeta + O^2))\right)^{-(2d-q)s'}\\
 &= (K(z,z)^{s-s'} + \sum_{j=1}^N \zeta_j g_j(z,\overline z) 
+ \sum^N_{j=1} \overline \zeta_j \overline{g_j(z,\overline z)} + O(\zeta^2))
 \left(|\zeta|^2 - K(z,z)^{-s}\right)
\end{aligned}
\end{equation}
where $\pi_1\colon D\rightarrow \mathbb C^n$ and $\pi_2\colon D\rightarrow \mathbb C^m$
are projections.
In the first line of \eqref{equation2}, the second order term in $\zeta$ variable is 
$|c\zeta|^2$ which does not vanish by the condition \eqref{condition 3}.
On the other hand in the second line of \eqref{equation2}, 
the second order term in $\zeta$ variable exists only when $K(z,z)^{s-s'}|\zeta|^2 = |\zeta|^2$.
That is $s=s'$. Moreover $||c\zeta|| = ||\zeta||$ and hence we obtain that $c$ is a unitary transformation.
Furthermore, if we consider the $|z|^2$ term, it follows that $a=0$.
\end{proof}
\begin{remark}
In the proof of Theorem \ref{s=s'}, we used the condition $s\geq 1$ in the definition of 
the Hartogs domains.
\end{remark}
\begin{lemma}\cite[Lemma 2.1]{Tumanov}\label{Tumanov}
Let $f \colon D\rightarrow D'$ be a proper holomorphic map of bounded domains in $\mathbb C^n$
each of whose kernel functions $K_D$ and $K_{D'}$ extends to the boundary.
Assume that there is a sequence of points $z_k\in D$, $k=1,2,\ldots$, satisfying the following conditions:
\begin{enumerate}
\item $z\rightarrow a\in \partial D$ and $f(z_k)\rightarrow b\in \partial D'$ as $k\rightarrow \infty$.
\item $|Jf(z_k)| > \epsilon$ for some $\epsilon >0$.
\item There exists a point $p\in D'$ such that $K'(b,p)\neq 0$ and $R'(b,p)\neq 0$,
where 
$$
R'(z,\zeta) = \det \left( \frac{\partial^2}{\partial z_i\overline \partial\zeta_j}\log K_{D'}(z,\zeta)\right).
$$
\end{enumerate}
Then $f$ extends holomorphically to a neighborhood of $a\in \partial D$.
\end{lemma}

\begin{proof}[{\bf Proof of Theorem \ref{s=s'}}]
Suppose that $D$ is the unit ball. Then $D_{N,s}$ is an ellipsoid.
Suppose that there is a biholomorphic map $f \colon D_{N,s}\rightarrow D'_{N',s'}$.
If $D'$ is not biholomorphic to the unit ball, then there is an embedding 
from $\Delta\times M$ for some complex submanifold $M$ into $D'$ and hence 
there should be a complex submanifold of dimension greater than zero in $\partial D_{N,s}$ 
by the same reason of Lemma \ref{center to center}.
This contradiction implies that $D'$ is also the unit ball. Hence $D'_{N',s'}$ is an ellipsoid.
In this case the theorem is proved by Naruki in \cite{Naruki}.

Now suppose that $D$ and $D'$ are not the unit balls.
By Lemma \ref{center to center} it follows that 
$f|_{D\times \{0\} }\colon D\times \{0\} \rightarrow D'\times \{0\}$
 is a biholomorphism and by Theorem \ref{Tumanov} it is rational.
Hence $N=N'$ and there exists a sequence $z_k\in D\times \{0\}$ 
such that $z_k\rightarrow  S\times \{0\}$
and $Jf(z_k) >\epsilon$ for some $\epsilon>0$.
Say $a:= \lim_{k\rightarrow \infty} z_k$.
Here $S$ denotes the Shilov boundary of $D$.
By Lemma \ref{Tumanov} $f$ extends 
holomorphically to a neighborhood of $a$.

By Lemma \ref{Shilov boundary} it follows that $b :=\lim_{k\rightarrow \infty} f(z_k)$ is also 
contained in the Shilov boundary of $D'$.
Therefore there exist
$\phi\in \text{Aut}(D)$ and $\psi\in \text{Aut}(D')$ such that 
$\phi(a)=0$ and $\psi(b) = 0$.
Denote $\Phi\in \text{Aut}(D_{N,s})$ and $\Psi\in \text{Aut}(D'_{N',s'})$ induced by $\phi$ and $\psi$ according to \eqref{induced auto} respectively. 
Then $\Psi\circ f\circ \Phi^{-1}$ is a biholomorphism from $D_{N,s}$ to $D'_{N,s'}$ which extends
holomorphically to a neighborhood of $(0,0)$.

Let $\sigma := (\Psi\circ f\circ \Phi^{-1})^{-1}|_{D'\times\{0\}}$. 
Note that it is a biholomorphism from $D'$ to $D$.
Let $\Sigma$ be a biholomorphism from $D'_{N,s'}$ onto $D_{N,s'}$ defined by \eqref{biholo map}.
Then $\Sigma\circ \Psi\circ f\circ \Phi^{-1} \colon D_{N,s}\rightarrow D_{N,s'}$ is a 
biholomorphism satisfying the conditions in Proposition \ref{s=s'} and hence $s=s'$
and hence $\Sigma\circ \Psi\circ f\circ \Phi^{-1}$ is an automorphism of $D_{N,s}$.

Note that $\Sigma\circ \Psi\circ f\circ \Phi^{-1}(z,0)=(z,0)$ for any $z\in D$.
By Proposition \ref{rigidity of holo}, it follows that 
\begin{equation}
d(\Sigma\circ \Psi\circ f\circ \Phi^{-1})|_{(0,0)} = \left(
  \begin{array}{cc}
  I_r& b\\
  0 & c
  \end{array}\right)
 \end{equation}
 where $c$ is a $m\times m$ unitary matrix.
 Define $C(z,\zeta) = (z, c^{-1}\zeta)$ and 
$F := C\circ \Sigma\circ \Psi\circ f\circ \Phi^{-1}$. 
 Then \begin{equation}
 dF|_{(0, 0)} = \left(
 \begin{array}{cc}
 I_r& b\\
 0& I_N
 \end{array}\right)
 \end{equation} 
On the other hand $k$-th composition of $ F$, say $F^k$, 
is also an automorphism
and there is an automorphism $\tilde F$ of $D_{N,s}$ such that $ F^k$ converges uniformly in compact-open topology 
since $ F|_{D\times \{0\}}$ is the identity map.
By Lemma \ref{Tumanov} $\tilde F$ extends holomorphically over the boundary
and hence $d\tilde F|_{(0,0)}$ should be finite.
This implies that $b=0$.

Expanding $ F$ in a power series at $(0,0)$ yields
$$ F(\xi) = \xi + P_k(\xi) + O(|\xi|^{k+1}),$$
where $P_k$ is the first nonvanishing homogeneous polynomial of degree k
of order exceeding $1$ in the Taylor expansion.
Then direct computation gives that 
\begin{equation}
\begin{aligned}
 F^2(\xi) &= \xi + 2 P_k(\xi) + O(|\xi|^{k+1})\\
&\,\,\,\vdots\\
 F^j(\xi) &= \xi + j P_k(\xi) + O(|\xi|^{k+1})
\end{aligned}
\end{equation}
where $\xi = (z,\zeta)$.
But this gives a contradiction to that $ \tilde F$ extends to a neighborhood 
of $(0,0)$.
This implies that $F = id$ on $D_{N,s}$ and hence 
the theorem is proved.
\end{proof}

\begin{proof}[Proof of Theorem \ref{proper map}]
Because of Lemma \ref{product}, it follows that $f$ maps $D\times \{0\}$ into $D'\times \{0\}$.
Let $H = \{z \in D_{N,s} : Jf(z)=0\}$. 
Suppose that there is a sequence $z_k$ in $H$ such that $z_k\rightarrow \partial _0D_{N,s}$
and $f(z_k)\rightarrow \partial_0 D'_{N',s'}$ as $k\rightarrow \infty$.
Since $\partial_0 D_{N,s}$ and $\partial_0 D'_{N',s'}$ are strongly pseudoconvex by Lemma 
\ref{strongly pseudoconvex}, we may take scaling sequences 
$s_k \colon D_{N,s}\rightarrow \mathbb C^d$
and $S_k \colon D'_{N',s'} \rightarrow \mathbb C^d$ where $d=\dim D_{N,s}=\dim D'_{N',s'}$
such that $s_k$ and $S_k$ are biholomorphisms onto their images and 
$s_k(D_{N,s})$, $S_k(D'_{N',s'})$ converge to the unit ball in $\mathbb C^d$
(For more detail, see \cite[9.2 Higher Dimensional Scaling and the Wong-Rosay Theorem]{Greene_Kim_Krantz}).
Then $S_k\circ f\circ s_k^{-1}$ converges a proper holomorphic map between the unit ball
in $\mathbb C^d$, say $F$. Since every proper holomorphic self-map of the unit ball 
is an automorphism (\cite{Alexander}), $F$ is an automorphism of the unit ball
and hence extends to the closure of the unit ball as a diffeomorphism.
This contradicts to that $z_k\in H$ and hence $f(H)\subset D'\times \{0\}$ by the maximum principle.
Since the codimension of $f(H)$ is $1$, there is an open subset in $D$ such that $Jf$ vanishes.
It is a contradiction to the fact that $f$ is a proper holomorphic map. Hence $H=\emptyset$,
i.e., $f$ is unbranched.
Since $D_{N,s}$ is simply connected, $f$ is a biholomorphism.

\end{proof}

\begin{remark}
When $D$ and $D'$ are the unit balls, that is, when $D_{N,s}$ and $D'_{N',s'}$ are 
complex ellipsoids, there is a proper holomorphic map which is not a biholomorphism.
\end{remark}


\begin{thebibliography}{}

\bibitem{Ahn_Byun_Park}
Ahn, Heungju; Byun, Jisoo; Park, Jong-Do, {\it Automorphisms of the Hartogs type domains over classical symmetric domains}, Internat. J. Math. {\it 23}  (2012),  no. 9, 1250098, 11 pp. 

\bibitem{Alexander}  Alexander, H., {\it Holomorphic mappings from the ball and polydisc}, Math. Ann. \textbf{209} (1974), 249--256.

\bibitem{Bell}
Bell, Steven R., {\it Proper holomorphic mappings between circular domains}, Comment. Math. Helv.  {\it 57}  (1982), no. 4, 532--538.

\bibitem{Gindikin}
 Gindikin, S. G., {\it Analysis in homogeneous domains}, (Russian) Uspehi Mat. Nauk 19. {\bf 118}(1964) no. 4, 3--92. 
\bibitem{Greene_Kim_Krantz}
 Greene, Robert E.; Kim, Kang-Tae; Krantz, Steven G., {\it The geometry of complex domains}, Progress in Mathematics, 291. Birkhäuser Boston, Inc., Boston, MA, 2011. xiv+303 pp. ISBN: 978-0-8176-4139-9 

\bibitem{Ishi_Park_Yamamori} Ishi, Hideyuki; Park, Jong-Do; Yamamori, Atsushi, {\it Bergman kernel function for Hartogs domains over bounded homogeneous domains}, J. Geom. Anal. {\bf 27} (2017),  no. 2, 1703--1736. 


\bibitem{Kaneyuki}
Kaneyuki, Soji, {\it Homogeneous bounded domains and Siegel domains}, Lecture Notes in Mathematics, Vol. {\bf 241}. Springer-Verlag, Berlin-New York, 1971. v+89 pp.

\bibitem{Kodama}
Kodama, A., {\it On the holomorphic automorphism group of a generalized complex ellipsoid}, Complex
Var. Elliptic Equ. {\bf 59}, 1342--1349 (2014)
\bibitem{Mok_Tsai}  Mok, Ngaiming; Tsai, I Hsun Rigidity of convex realizations of irreducible bounded symmetric domains of rank $\geq$ 2. J. Reine Angew. Math. {\bf 431}  (1992), 91--122.


\bibitem{Naruki}
Naruki, I., {\it The holomorphic equivalence problem for a class of Reinhardt domains}, Publ. Res. Inst.
Math. Sci. Kyoto Univ. {\bf 4}, 527--543 (1968)




\bibitem{Pya}
Pyateskii-Shapiro, I. I., {\it Automorphic functions and the geometry of classical domains}, 
Translated from the Russian. Mathematics and Its Applications, Vol. {\bf 8} Gordon and Breach Science Publishers, New York-London-Paris 1969 viii+264 pp.

\bibitem{Rudin}
W. Rudin, {\it Function theory in polydiscs}, W. A. Benjamin, Inc., New York-Amsterdam 1969 vii+188 pp. 


\bibitem{Tu_Wang}
Tu, Zhenhan; Wang, Lei, {\it Rigidity of proper holomorphic mappings between equidimensional Hua domains}, Math. Ann. {\bf 363}  (2015),  no. 1--2, 1--34.

\bibitem{Tumanov}
Tumanov, A. E., {\it Finite-dimensionality of the group of CR-automorphisms of a standard CR-manifold, and characteristic holomorphic mappings of Siegel domains}, (Russian) Izv. Akad. Nauk SSSR Ser. Mat. {\bf 52} (1988), no. 3, 651--659, 672; translation in  Math. USSR-Izv. {\bf 32} (1989), no. 3, 655--662



\bibitem{Zhao_Wang_Hao}
Zhao, Jie; Wang, An; Hao, Yihong, {\it On the holomorphic automorphism group of the Bergman-Hartogs domain}, Internat. J. Math. {\bf 26}  (2015),  no. 8, 1550056, 13 pp.
\end{thebibliography}
\end{document}